\documentclass[10pt]{amsart}%
\usepackage{graphicx}
\usepackage{amscd, color}
\usepackage{amsmath}
\usepackage{amsfonts}
\usepackage{amssymb}%
\providecommand{\U}[1]{\protect\rule{.1in}{.1in}}
\providecommand{\U}[1]{\protect\rule{.1in}{.1in}}
\providecommand{\U}[1]{\protect\rule{.1in}{.1in}} \textwidth 16.3cm
\theoremstyle{plain}

\newtheorem{theorem}{Theorem}[section]
\newtheorem{proposition}[theorem]{Proposition}

\newtheorem{remark}[theorem]{Remark}

\newtheorem{lemma}[theorem]{Lemma}

\newtheorem{definition}[theorem]{Definition}
\numberwithin{equation}{section}

\begin{document}
\title[Almost summing polynomials]{Almost summing polynomials: a coherent and compatible approach the from the
infinite-dimensional holomorphy viewpoint}
\author{Daniel Pellegrino and Joilson Ribeiro}
\address[D. Pellegrino]{ Departamento de Matem\'{a}tica, Universidade Federal da
Para\'{\i}ba, 58.051-900 - Jo\~{a}o Pessoa, Brazil\\
[J. Ribeiro] Departamento de Matem\'{a}tica - CCEN - UFPE, Av. Prof. Luiz
Freire, s/n., Cidade Universit\'{a}ria, 50740-540, Recife, Brazil}
\email{pellegrino@pq.cnpq.br}
\thanks{D. Pellegrino is supported by CNPq Grant 301237/2009-3. }

\begin{abstract}
In this note we explore the notion of everywhere almost summing polynomials
and define a natural norm which makes this class a Banach multi-ideal which is
a holomorphy type (in the sense of L. Nachbin) and also coherent and
compatible (in the sense of D. Carando, V. Dimant and S. Muro) with the notion
of almost summing linear operators. Similar results are not valid for the
original concept of almost summing polynomials, due to G. Botelho.

\end{abstract}
\maketitle

\section{Introduction}

An operator ideal $\mathcal{I}$ is a subclass of the class $\mathcal{L}$ of
all continuous linear operators between Banach spaces such that for all Banach
spaces $E$ and $F$ its components%
\[
\mathcal{I}(E;F):=\mathcal{L}(E;F)\cap\mathcal{I}%
\]
satisfy:

(a) $\mathcal{I}(E;F)$ is a linear subspace of $\mathcal{L}(E;F)$ which
contains the finite rank operators.

(b) If $u\in\mathcal{I}(E;F)$, $v\in\mathcal{L}(G;E)$ and $w\in\mathcal{L}%
(F;H)$, then $w\circ u\circ v\in\mathcal{I}(G;H)$.

The operator ideal is a normed operator ideal if there is a function
$\Vert\cdot\Vert_{\mathcal{I}}\colon\mathcal{I}\longrightarrow\lbrack
0,\infty)$ satisfying

\bigskip

(1) $\Vert\cdot\Vert_{\mathcal{I}}$ restricted to $\mathcal{I}(E;F)$ is a
norm, for all Banach spaces $E$, $F$.

(2) $\Vert id\colon\mathbb{K}\longrightarrow\mathbb{K}:id(\lambda
)=\lambda\Vert_{\mathcal{I}}=1$,

(3) If $u\in\mathcal{I}(E;F)$, $v\in\mathcal{L}(G;E)$ and $w\in\mathcal{L}%
(F;H)$, then $\Vert w\circ u\circ v\Vert_{\mathcal{I}}\leq\Vert w\Vert\Vert
u\Vert_{\mathcal{I}}\Vert v\Vert$.

\bigskip

When $\mathcal{I}(E;F)$ with the above norm is always complete, $\mathcal{I}$
is called a Banach operator ideal. For details we refer to \cite{djp, mono}.

\bigskip

The notion of operator ideals, as well as its generalization to the
multilinear setting is due to A. Pietsch \cite{mono, PPPP, pp22}\ (we soon
will present the precise definitions). A natural question is: given an
operator ideal $\mathcal{I}$ (for example the class of compact operators) how
to define a multi-ideal and a polynomial ideal that keeps the spirit of
$\mathcal{I}$ ?

Abstract methods of defining when multilinear (and polynomial) extensions of
operator ideals are, in some sense, compatible with the structure of the
linear ideal were recently discussed by several authors. Sometimes a given
operator ideal has several possible extensions to multilinear and polynomial
ideals (we mention the case of absolutely summing operators \cite{PS,
davidarchiv} and almost summing operators \cite{BBJ-Archiv, JMAA}).

Some methods of evaluating multilinear/polynomial extensions of a given
operator ideal have been introduced recently. The idea is that, given positive
integers $k_{1}$ and $k_{2}$, the respective levels of $k_{1}$-linearity and
$k_{2}$-linearity of a given multi-ideal (or polynomial ideal) must have some
strong connection and also a connection with the original linear ideal
$(k=1)$. In this note we will be interested in the notions of holomorphy types
(in the sense of L. Nachbin \cite{Nachbin}) and also coherence and
compatibility (in the sense of D. Carando, V. Dimant and S. Muro \cite{CDM09,
ccc, ccc2}). All these concepts are related to the context of
infinite-dimensional holomorphy and for this reason we deal with the case
$\mathbb{K}=\mathbb{C}$, but the results can be naturally extended to the case
of real scalars. For the theory of polynomials in Banach spaces and
infinite-dimensional holomorphy we refer to \cite{din, Mu}.

The main goal of this note is to show that the space of everywhere almost
summing multilinear operators can be normed in such a way that it is a Banach
multi-ideal and the associated notion of everywhere almost summing polynomials
will be a (global) holomorphy type and also coherent and compatible (in the
sense of \cite{CDM09}) with the notion of almost summing linear operator. We
also show that the previous approach to almost summing polynomials fails these properties.

\section{Almost summing operators and polynomials}

The idea of considering the notion of nonlinear almost summing operators
appears in \cite{Botelho-Nach, BBJ-Archiv} and was also explored in
\cite{Daniel-Archiv, JMAA, PS}. For the linear theory of almost summing
operators we refer to classical monograph \cite{Diestel}. In this work we aim
to explore the concept of almost summing mapping at a given point, obtaining a
norm with good properties in the space of almost summing mappings at every
point (called everywhere almost summing mappings). Initially we, in part,
adapted similar results (for absolutely summing maps) obtained in \cite{Scand,
BBJP, Matos-N}; however, the case of almost summing mappings is more delicate,
and even in the linear case requires special attention (see, \cite[Pag.
27]{Botelho-Nach} ).

Throughout this paper $E,E_{1},\ldots,E_{n},F,G,G_{1},\ldots,G_{n},H$ will
stand for Banach spaces; $B_{E}$ denotes the closed unit ball of $E$ and
$E^{\ast}$ denotes the topological dual of $E$.

For $1\leq p<\infty$, the Banach space of all sequences $\left(  x_{j}\right)
_{j=1}^{\infty}$ in $E$ such that
\[
\left\Vert \left(  x_{j}\right)  _{j=1}^{\infty}\right\Vert _{p}=\left(
\sum_{j=1}^{\infty}\left\Vert x_{j}\right\Vert ^{p}\right)  ^{\frac{1}{p}%
}<\infty
\]
is represented by $\ell_{p}\left(  E\right)  $. We denote by $\ell_{p}%
^{w}\left(  E\right)  $ the linear space formed by the sequences $\left(
x_{j}\right)  _{j=1}^{\infty}$ in $E$ such that $\left(  \varphi\left(
x_{j}\right)  \right)  _{j=1}^{\infty}\in\ell_{p}\left(  \mathbb{K}\right)  $,
for every $\varphi\in E^{\ast}$. The space $\ell_{p}^{w}\left(  E\right)
$\ is a Banach space when endowed with the norm
\[
\left\Vert \left(  x_{j}\right)  _{j=1}^{\infty}\right\Vert _{w,p}%
:=\sup_{\varphi\in B_{E^{\ast}}}\left(  \sum_{j=1}^{\infty}\left\vert
\varphi\left(  x_{j}\right)  \right\vert ^{p}\right)  ^{\frac{1}{p}}.
\]
The closed linear subspace of $\ell_{p}^{w}\left(  E\right)  $ of all
sequences $\left(  x_{j}\right)  _{j=1}^{\infty}$ in $\ell_{p}^{w}\left(
E\right)  $, such that
\[
\lim_{m\rightarrow\infty}\left\Vert \left(  x_{j}\right)  _{j=m}^{\infty
}\right\Vert _{w,p}=0
\]
is denoted by $\ell_{p}^{u}\left(  E\right)  $.

For $0<p<\infty$, $Rad_{p}\left(  E\right)  $ represents the vector space
formed by the sequences $\left(  x_{j}\right)  _{j=1}^{\infty}$ so that
\[
\sum_{j=1}^{n}r_{j}\left(  \cdot\right)  x_{j}%
\]
converges in $L_{p}\left(  E\right)  =L_{p}\left(  \left[  0,1\right]
;E\right)  $, where $r_{j}$ denotes the $j^{\text{th }}$-Rademacher function
(or, equivalently, if $\sum_{j=1}^{n}r_{j}\left(  t\right)  x_{j}$ is
convergent in $E$\ for (Lebesgue) almost all $t\in\left[  0,1\right]  $). It
is well-known (Kahane's Inequality (cf. \cite[11.1]{Diestel})) that
$Rad_{p}\left(  E\right)  =Rad_{q}\left(  E\right)  $ for all $0<p,q<\infty$
and for this reason we just write $Rad\left(  E\right)  $ instead of
$Rad_{p}\left(  E\right)  $. The space $Rad\left(  E\right)  $ is a Banach
space when endowed with a norm given by
\[
\left\Vert \left(  x_{j}\right)  _{j=1}^{\infty}\right\Vert _{Rad\left(
E\right)  }:=\left(  \int_{0}^{1}\left\Vert \sum_{j=1}^{\infty}r_{j}\left(
t\right)  x_{j}\right\Vert ^{2}dt\right)  ^{\frac{1}{2}}.
\]
If $1<p\leq2$, a continuous linear operator $u:E\rightarrow F$ is almost
$p$-summing when $\left(  u(x_{j})\right)  _{j=1}^{\infty}\in Rad(F)$ whenever
$(x_{j})_{j=1}^{\infty}\in\ell_{p}^{u}(E)$. The theory of almost summing
operators has a strong connection with the well-known theory of absolutely
summing operators. For details on almost summing operators we refer to the
excellent monograph \cite{Diestel} and for recent results we mention
\cite{Botelho-Nach, Daniel-Archiv, ta} and for classical results on absolutely
summing linear operators we refer to \cite{di, Diestel} and references
therein; for recent results on linear (and nonlinear) absolutely summing
operators we refer to \cite{arregui, badea, PellZ, cilia, Def2, junek, ku,
pesa, pesa2} and to \cite{resumeD} for a modern approach to Grothendieck's
Resum\'{e} and the roots of the theory of absolutely summing operators. The
class of almost $p$-summing operators, denoted by $\Pi_{al,p}$ (with a natural
norm that will be clear in the forthcoming Remark \ref{yb}), is Banach
operator ideal.

The concept of multi-ideals is also due to A. Pietsch \cite{PPPP}. An ideal of
multilinear mappings (or multi-ideal) $\mathcal{M}$ is a subclass of the class
$\mathcal{L}$ of all continuous multilinear operators between Banach spaces
such that for any positive integer $n$, Banach spaces $E_{1},\ldots,E_{n}$ and
$F$, the components
\[
\mathcal{M}(E_{1},\ldots,E_{n};F):=\mathcal{L}(E_{1},\ldots,E_{n}%
;F)\cap\mathcal{M}%
\]
satisfy:

\bigskip

(a) $\mathcal{M}(E_{1},\ldots,E_{n};F)$ is a linear subspace of $\mathcal{L}%
(E_{1},\ldots,E_{n};F)$ which contains the $n$-linear mappings of finite type.

(b) If $T\in\mathcal{M}(E_{1},\ldots,E_{n};F)$, $u_{j}\in\mathcal{L}%
(G_{j};E_{j})$ for $j=1,\ldots,n$ and $v\in\mathcal{L}(F;H)$, then $v\circ
T\circ(u_{1},\ldots,u_{n})$ belongs to $\mathcal{M}(G_{1},\ldots,G_{n};H)$.

Moreover, $\mathcal{M}$ is a normed multi-ideal if there is a function
$\Vert\cdot\Vert_{\mathcal{M}}\colon\mathcal{M}\longrightarrow\lbrack
0,\infty)$ satisfying

\bigskip

(1) For each $n$, $\Vert\cdot\Vert_{\mathcal{M}}$ restricted to $\mathcal{M}%
(E_{1},\ldots,E_{n};F)$ is a norm, for all Banach spaces $E_{1},\ldots,E_{n}$
and $F.$

(2) $\Vert A_{n}\colon\mathbb{K}^{n}\longrightarrow\mathbb{K}:A_{n}%
(\lambda_{1},\ldots,\lambda_{n})=\lambda_{1}\cdots\lambda_{n}\Vert
_{\mathcal{M}}=1$ for all $n$,

(3) If $T\in\mathcal{M}(E_{1},\ldots,E_{n};F)$, $u_{j}\in\mathcal{L}%
(G_{j};E_{j})$ for $j=1,\ldots,n$ and $v\in\mathcal{L}(F;H)$, then
\[
\Vert v\circ T\circ(u_{1},\ldots,u_{n})\Vert_{\mathcal{M}}\leq\Vert
v\Vert\Vert T\Vert_{\mathcal{M}}\Vert u_{1}\Vert\cdots\Vert u_{n}\Vert.
\]

Analogously, a polynomial ideal is a class $\mathcal{Q}$ of continuous
homogeneous polynomials between Banach spaces so that for all $n\in\mathbb{N}$
and all Banach spaces $E$ and $F$, the components
\[
\mathcal{Q}\left(  ^{n}E;F\right)  :=\mathcal{P}\left(  ^{n}E;F\right)
\cap\mathcal{Q}%
\]
satisfy

(a) $\mathcal{Q}\left(  ^{n}E;F\right)  $ is a subspace of $\mathcal{P}\left(
^{n}E;F\right)  $ which contains the finite-type polynomials.

(b) If $u\in\mathcal{L}\left(  G;E\right)  $, $P\in\mathcal{Q}\left(
^{n}E;F\right)  $ and $w\in\mathcal{L}\left(  ^{n}E;F\right)  $, then
\[
w\circ P\circ u\in\mathcal{Q}\left(  ^{n}G;H\right)  .
\]

If there exists a map $\left\Vert \cdot\right\Vert _{\mathcal{Q}}%
:\mathcal{Q}\rightarrow\lbrack0,\infty\lbrack$ satisfying

(1) For each $n,$ $\left\Vert \cdot\right\Vert _{\mathcal{Q}}$ restricted to
$\mathcal{Q}(^{n}E;F)$ is a norm for all Banach spaces $E$ and $F$;

(2) $\left\Vert P_{n}:\mathbb{K}\rightarrow\mathbb{K};\text{ }P_{n}\left(
\lambda\right)  =\lambda^{n}\right\Vert _{\mathcal{Q}}=1$ for all $n$;

(3) If $u\in\mathcal{L}(G;E)$, $P\in\mathcal{Q}(^{n}E;F)$ and $w\in
\mathcal{L}(F;H),$ then
\[
\left\Vert w\circ P\circ u\right\Vert _{\mathcal{Q}}\leq\left\Vert
w\right\Vert \left\Vert P\right\Vert _{\mathcal{Q}}\left\Vert u\right\Vert
^{n},
\]
then $\mathcal{Q}$ is a normed polynomial ideal. If all components
$\mathcal{Q}\left(  ^{n}E;F\right)  $ are complete, $\left(  \mathcal{Q}%
,\left\Vert \cdot\right\Vert _{\mathcal{Q}}\right)  $ is called a Banach-ideal
of polynomials.

Let $1<p\leq2.$ Originally, as defined in \cite{Botelho-Nach}, a polynomial
$P\in\mathcal{P}\left(  ^{n}E;F\right)  $ is almost summing if $\left(
P\left(  x_{j}\right)  \right)  _{j=1}^{\infty}\in Rad\left(  F\right)  $
whenever $\left(  x_{j}\right)  _{j=1}^{\infty}\in\ell_{p}^{u}\left(
E\right)  .$ However, this concept (with any norm) fails to be a global
holomorphy type (for the definition of global holomorphy type we refer to
\cite{BBJP}). For example, let $P\in\mathcal{P}(^{2}c_{0};c_{0})$ be given by
\[
P\left(  x\right)  =\varphi\left(  x\right)  x,
\]
where $0\neq\varphi\in E^{\ast},$ $a\in c_{0}$ and $\varphi\left(  a\right)
=1$. From \cite[Corollary 3]{Daniel-Archiv} we have $P\in\mathcal{P}%
_{al,2}(^{2}c_{0};c_{0})$ but $dP\left(  a\right)  \notin\mathcal{L}%
_{al,2}(c_{0};c_{0})$ and hence $\mathcal{P}_{al,2}$ is not a global
holomorphy type. The same example also shows that this concept is not
compatible with the ideal of almost summing operators in the sense of
\cite{CDM09}.

In \cite{Daniel-Archiv} a new approach to almost summability of arbitrarily
nonlinear mappings was presented, where almost summability was considered on
certain points of the domain, as we will see in the next section.

\section{Everywhere almost summing multilinear mappings and polynomials}

Let $1<p_{1},...,p_{n}\leq2$. A map $T\in\mathcal{L}(E_{1},\ldots,E_{n};F)$ is
almost $\left(  p_{1},\ldots,p_{n}\right)  $-summing at $a=\left(
a_{1},\ldots,a_{n}\right)  \in E_{1}\times\cdots\times E_{n}$ if
\[
\left(  T\left(  a_{1}+x_{j}^{\left(  1\right)  },\ldots,a_{n}+x_{j}^{\left(
n\right)  }\right)  -T\left(  a_{1},\ldots,a_{n}\right)  \right)
_{j=1}^{\infty}\in Rad\left(  F\right)
\]
whenever $\left(  x_{j}^{\left(  i\right)  }\right)  _{j=1}^{\infty}\in
\ell_{p_{i}}^{u}\left(  E_{i}\right)  $, $i=1,\ldots,n$. The set of all
$T\in\mathcal{L}(E_{1},\ldots,E_{n};F)$ which are almost $\left(  p_{1}%
,\ldots,p_{n}\right)  $-summing at $a=\left(  a_{1},\ldots,a_{n}\right)  $ is
denoted by $\mathcal{L}_{al,\left(  p_{1},\ldots,p_{n}\right)  }^{(a)}%
(E_{1},\ldots,E_{n};F)$.

In particular, if $T$ is almost $\left(  p_{1},\ldots,p_{n}\right)  $-summing
at $a=0$, we say that $T$ is almost $\left(  p_{1},\ldots,p_{n}\right)
$-summing and we denote the respective set by $\mathcal{L}_{al,\left(
p_{1},\ldots,p_{n}\right)  }(E_{1},\ldots,E_{n};F)$.

If $E_{1}=\ldots=E_{n}=E$ and $p_{1}=...=p_{n}=p$ we write $\mathcal{L}%
_{al,p}^{(a)}(^{n}E;F)$ or $\mathcal{L}_{al,p}(^{n}E;F)$ if $a=0.$

The space composed by all continuous $n$-linear mappings which are $\left(
p_{1},\ldots,p_{n}\right)  $-summing at every point is denoted by
$\mathcal{L}_{al,\left(  p_{1},\ldots,p_{n}\right)  }^{ev}(E_{1},\ldots
,E_{n};F).$ A similar definition holds for polynomials and the respective
spaces are represented by $\mathcal{P}_{al,p}^{(a)}(^{n}E;F)$ and
$\mathcal{P}_{al,p}^{ev}(^{n}E;F)$.

Some of the arguments used in the forthcoming proofs related to
multilinear/polynomial notions of almost summing operators can be obtained,
\textit{mutatis mutandis}, from similar arguments used for absolutely summing
operators in \cite{Scand, Matos-N}. For this reason we will omit the more
simple adaptations and concentrate on the parts that need a different treatment.

\subsection{Dvoretzky-Rogers type theorems}

The following result was proved in \cite[Theorem 4]{Daniel-Archiv}

\begin{theorem}
\label{yya}For $1<p\leq2,$%
\begin{align*}
\mathcal{P}(^{n}E;E)  &  \neq\mathcal{P}_{al,p}^{ev}(^{n}E;E)\Leftrightarrow
\dim E=\infty\text{ and }\\
\mathcal{L}(^{n}E;E)  &  \neq\mathcal{L}_{al,p}^{ev}(^{n}E;E)\Leftrightarrow
\dim E=\infty.
\end{align*}

\end{theorem}

\bigskip A repetition of the arguments used in \cite[Theorems 3.2 and
3.7]{Scand} furnishes the following improvements of Theorem \ref{yya}:

\begin{theorem}
[Dvoretzky-Rogers type Theorem for multilinear operators]\label{kkkl} Let $E$
be a Banach space, $n\geq2$ and $1<p\leq2$. The following assertions are equivalent:

\textrm{(a)} $E$ is infinite-dimensional.

\textrm{(b)} $\mathcal{L}_{al,p}^{(a)}(^{n}E;E)\neq\mathcal{L}(^{n}E;E)$ for
every $a=(a_{1},\ldots,a_{n})\in E^{n}$ with either $a_{i}\neq0$ for every $i$
or $a_{i}=0$ for only one $i$.

\textrm{(c)} $\mathcal{L}_{al,p}^{(a)}(^{n}E;E)\neq\mathcal{L}(^{n}E;E)$ for
some $a=(a_{1},\ldots,a_{n})\in E^{n}$ with either $a_{i}\neq0$ for every $i$
or $a_{i}=0$ for only one $i$.
\end{theorem}

\begin{theorem}
[Dvoretzky-Rogers type Theorem for polynomials]Let $E$ be a Banach space,
$n\geq2$ and $1<p\leq2$. The following assertions are equivalent:

\textrm{(a)} $E$ is infinite-dimensional.

\textrm{(b)} $\mathcal{P}_{al,p}^{(a)}(^{n}E;E)\neq\mathcal{P}(^{n}E;E)$ for
all $a\in E$, $a\neq0$.

\textrm{(c)} $\mathcal{P}_{al,p}^{(a)}(^{n}E;E)\neq\mathcal{P}(^{n}E;E)$ for
some $a\in E$, $a\neq0$.
\end{theorem}

\subsection{A characterization for the space of everywhere almost summing
mappings}

In order to define an adequate norm on the space $\mathcal{L}_{al,p}%
^{ev}(E_{1},...,E_{n};F)$ we will characterize the maps $T\in\mathcal{L}%
_{al,p}^{ev}(E_{1},...,E_{n};F)$ by means of an inequality (Theorem
\ref{teorema-Caracterizacao-L(al,p,ev)}).

We will need the Contraction Principle:

\begin{theorem}
[{\cite[12.2 - Contraction Principle]{Diestel}}]\label{ContractionPrinciple}
Let $1\leq p<\infty$ and consider the randomized sum $\sum_{k=1}^{n}\chi
_{k}x_{k}$ in the Banach space $E$. Then, regardless of the choice of real
numbers $a_{1},\ldots,a_{n},$%
\[
\left(  \int_{\Omega}\left\Vert \sum_{k\leq n}a_{k}\chi_{k}\left(
\omega\right)  x_{k}\right\Vert ^{p}dP\left(  \omega\right)  \right)
^{\frac{1}{p}}\leq\left(  \max_{k\leq n}\left\vert a_{k}\right\vert \right)
.\left(  \int_{\Omega}\left\Vert \sum_{k\leq n}\chi_{k}\left(  \omega\right)
x_{k}\right\Vert ^{p}dP\left(  \omega\right)  \right)  ^{\frac{1}{p}}.
\]
In particular, if $A$ and $B$ are subsets of $\left\{  1,\ldots,n\right\}  $
such that $A\subset B$, then
\[
\left(  \int_{\Omega}\left\Vert \sum_{k\in A}\chi_{k}\left(  \omega\right)
x_{k}\right\Vert ^{p}dP\left(  \omega\right)  \right)  ^{\frac{1}{p}}%
\leq\left(  \int_{\Omega}\left\Vert \sum_{k\in B}\chi_{k}\left(
\omega\right)  x_{k}\right\Vert ^{p}dP\left(  \omega\right)  \right)
^{\frac{1}{p}}%
\]
\bigskip
\end{theorem}

The following result is a simple consequence of the Contraction Principle:

\begin{lemma}
\label{ContracaoInfinito} If $\left(  x_{j}\right)  _{j=1}^{\infty}\in
Rad(F)$, then
\[
\left\Vert x_{n}\right\Vert \leq2\left\Vert \left(  x_{j}\right)
_{j=1}^{\infty}\right\Vert _{Rad(F)}%
\]
for all $n$.
\end{lemma}

\begin{proof}
From the Contraction Principle we know that
\[
\left(
{\displaystyle\int_{0}^{1}}
\left\Vert
{\displaystyle\sum\limits_{j=1}^{n}}
r_{j}(t)x_{j}\right\Vert ^{2}dt\right)  _{n=1}^{\infty}\uparrow%
{\displaystyle\int_{0}^{1}}
\left\Vert
{\displaystyle\sum\limits_{j=1}^{\infty}}
r_{j}(t)x_{j}\right\Vert ^{2}dt.
\]
So, for all $n$, we have%
\begin{align*}
\left\Vert x_{n+1}\right\Vert  &  =\left(
{\displaystyle\int_{0}^{1}}
\left\Vert
{\displaystyle\sum\limits_{j=1}^{n+1}}
r_{j}(t)x_{j}-%
{\displaystyle\sum\limits_{j=1}^{n}}
r_{j}(t)x_{j}\right\Vert ^{2}dt\right)  ^{\frac{1}{2}}\\
&  \leq\left(
{\displaystyle\int_{0}^{1}}
\left\Vert
{\displaystyle\sum\limits_{j=1}^{n+1}}
r_{j}(t)x_{j}\right\Vert ^{2}dt\right)  ^{\frac{1}{2}}+\left(
{\displaystyle\int_{0}^{1}}
\left\Vert
{\displaystyle\sum\limits_{j=1}^{n}}
r_{j}(t)x_{j}\right\Vert ^{2}dt\right)  ^{\frac{1}{2}}\leq2\left(
{\displaystyle\int_{0}^{1}}
\left\Vert
{\displaystyle\sum\limits_{j=1}^{\infty}}
r_{j}(t)x_{j}\right\Vert ^{2}dt\right)  ^{\frac{1}{2}}.
\end{align*}

\end{proof}

From the previous lemma we easily conclude that when $\left(  y_{j}%
^{(n)}\right)  _{j=1}^{\infty}\in Rad(F)$ and
\[
\lim_{n\rightarrow\infty}\left(  y_{j}^{(n)}\right)  _{j=1}^{\infty}%
=(y_{j})_{j=1}^{\infty}\in Rad(F),
\]
then
\[
\lim_{n\rightarrow\infty}y_{j}^{(n)}=y_{j}%
\]
for all $j$.

The following lemma, which proof can be obtained, \textit{mutatis mutandis},
from \cite[Lemma 9.2]{BBJP} is crucial for the proof of the main result of
this section:

\begin{lemma}
\label{tr}If $T\in\mathcal{L}_{al,p}^{ev}(E_{1},...,E_{n};F)$ and $\left(
a_{1},...,a_{n}\right)  \in E_{1}\times\cdots\times E_{n},$ then there is a
real number $C_{a_{1}...a_{n}}$ so that%
\[%
{\displaystyle\int_{0}^{1}}
\left\Vert
{\displaystyle\sum\limits_{j=1}^{\infty}}
r_{j}\left(  t\right)  \left(  T\left(  a_{1}+x_{j}^{(1)},...,a_{n}%
+x_{j}^{(n)}\right)  -T\left(  a_{1},...,a_{n}\right)  \right)  \right\Vert
^{2}dt\leq C_{a_{1}...a_{n}},
\]
whenever $\left(  x_{j}^{(k)}\right)  _{j=1}^{\infty}\in B_{\ell_{p}%
^{u}\left(  E_{k}\right)  },$ $k=1,...,n.$
\end{lemma}

From now on, if $T\in\mathcal{L}_{al,p}^{ev}(E_{1},...,E_{n};F),$ we consider
the $n$-linear map
\begin{align*}
\phi_{T}  &  :G_{1}\times\cdots\times G_{n}\rightarrow Rad(F)\\
\phi_{T}\left(  \left(  a_{1},\left(  x_{j}^{(1)}\right)  _{j=1}^{\infty
}\right)  ,...,\left(  a_{n},\left(  x_{j}^{(n)}\right)  _{j=1}^{\infty
}\right)  \right)   &  =\left(  T\left(  a_{1}+x_{j}^{(1)},...,a_{n}%
+x_{j}^{(n)}\right)  -T\left(  a_{1},...,a_{n}\right)  \right)  _{j=1}%
^{\infty},
\end{align*}
where
\[
G_{s}=E_{s}\times\ell_{p}^{u}\left(  E_{s}\right)  ,\text{ }s=1,..,n,
\]
is endowed with the norm%
\[
\left\Vert (a,(x_{j})_{j=1}^{\infty})\right\Vert _{G_{s}}=\max\left\{
\left\Vert a\right\Vert ,\left\Vert (x_{j})_{j=1}^{\infty}\right\Vert
_{w,p}\right\}  .
\]
The next result is inspired on \cite{Scand} and \cite{Matos-N}, but some parts
(specially the proof of (d)$\Rightarrow$(a)) need a different approach:

\begin{theorem}
\label{teorema-Caracterizacao-L(al,p,ev)} The following assertions are equivalent:

(a) $T\in\mathcal{L}_{al,p}^{ev}(E_{1},...,E_{n};F).$

(b) $\phi_{T}$ is well-defined and continuous.

(c) There is a $C\geq0$ so that%
\begin{equation}
\left(
{\displaystyle\int_{0}^{1}}
\left\Vert
{\displaystyle\sum\limits_{j=1}^{\infty}}
r_{j}\left(  t\right)  \left(  T\left(  a_{1}+x_{j}^{(1)},...,a_{n}%
+x_{j}^{(n)}\right)  -T\left(  a_{1},...,a_{n}\right)  \right)  \right\Vert
^{2}dt\right)  ^{\frac{1}{2}}\leq C%
{\displaystyle\prod\limits_{k=1}^{n}}
\left(  \left\Vert a_{k}\right\Vert +\left\Vert \left(  x_{j}^{(k)}\right)
_{j=1}^{\infty}\right\Vert _{w,p}\right) \nonumber
\end{equation}
for all $\left(  x_{j}^{(k)}\right)  _{j=1}^{\infty}\in\ell_{p}^{u}(E_{k}),$
$k=1,...,n$ and $(a_{1},...,a_{n})\in E_{1}\times\cdots\times E_{n}.$

(d) There is a $C\geq0$ so that%
\begin{equation}
\left(
{\displaystyle\int_{0}^{1}}
\left\Vert
{\displaystyle\sum\limits_{j=1}^{m}}
r_{j}\left(  t\right)  \left(  T\left(  a_{1}+x_{j}^{(1)},...,a_{n}%
+x_{j}^{(n)}\right)  -T\left(  a_{1},...,a_{n}\right)  \right)  \right\Vert
^{2}dt\right)  ^{\frac{1}{2}}\leq C%
{\displaystyle\prod\limits_{k=1}^{n}}
\left(  \left\Vert a_{k}\right\Vert +\left\Vert \left(  x_{j}^{(k)}\right)
_{j=1}^{m}\right\Vert _{w,p}\right)  \label{ptr}%
\end{equation}
for all positive integer $m,$ $x_{j}^{(k)}\in E_{k},$ $k=1,...,n$ e
$(a_{1},...,a_{n})\in E_{1}\times\cdots\times E_{n}.$
\end{theorem}

\begin{proof}
$(a)\Rightarrow(b)$\newline Let us show that $\phi_{T}$ is continuous. We
first show that the set%
\[
F_{k,\left(  x_{j}^{(1)}\right)  _{j=1}^{\infty},\ldots,\left(  x_{j}%
^{(n)}\right)  _{j=1}^{\infty}}=\left\{
\begin{array}
[c]{c}%
(a_{1},\ldots,a_{n})\in E_{1}\times\ldots\times E_{n};\\
\left\Vert \phi_{T}\left(  \left(  a_{1},\left(  x_{j}^{(1)}\right)
_{j=1}^{\infty}\right)  ,\ldots,\left(  a_{n},\left(  x_{j}^{(n)}\right)
_{j=1}^{\infty}\right)  \right)  \right\Vert _{Rad(F)}\leq k
\end{array}
\right\}
\]

is closed for all $k$ and $\left(  x_{j}^{(r)}\right)  _{j=1}^{\infty}\in
B_{\ell_{p}^{u}(E_{r})}$. For this task, consider the sets
\[
F_{k,\left(  x_{j}^{(1)}\right)  _{j=1}^{m},\ldots,\left(  x_{j}^{(n)}\right)
_{j=1}^{m}}=\left\{  (a_{1},\ldots,a_{n})\in E_{1}\times\ldots\times
E_{n};\left\Vert \varphi_{m}(T)\left(  a_{1},\ldots,a_{n}\right)  \right\Vert
_{Rad(F)}\leq k\right\}  ,
\]

where $\varphi_{m}(T)$ is the map%
\begin{align*}
\varphi_{m}\left(  T\right)   &  :E_{1}\times\cdots\times E_{n}\rightarrow
Rad(F)\\
\varphi_{m}\left(  T\right)  \left(  a_{1},\ldots,a_{n}\right)   &  =\left(
T\left(  a_{1}+x_{j}^{(1)},\ldots,a_{n}+x_{j}^{(n)}\right)  -T\left(
a_{1},\ldots,a_{n}\right)  \right)  _{j=1}^{m}.
\end{align*}
Using the continuity of $T$ we can show that $\varphi_{m}(T)$ is also
continuous. Hence
\[
F_{k,\left(  x_{j}^{(1)}\right)  _{j=1}^{m},\ldots,\left(  x_{j}^{(n)}\right)
_{j=1}^{m}}%
\]
is closed. Now we will show that
\[
F_{k,\left(  x_{j}^{(1)}\right)  _{j=1}^{\infty},\ldots,\left(  x_{j}%
^{(n)}\right)  _{j=1}^{\infty}}=\bigcap\limits_{m}F_{k,\left(  x_{j}%
^{(1)}\right)  _{j=1}^{m},\ldots,\left(  x_{j}^{(n)}\right)  _{j=1}^{m}}.
\]

If $(a_{1},\ldots,a_{n})\in F_{k,\left(  x_{j}^{(1)}\right)  _{j=1}^{\infty
},\ldots,\left(  x_{j}^{(n)}\right)  _{j=1}^{\infty}}$, then
\[
\left(  {\int_{0}^{1}}\left\Vert {\sum\limits_{j=1}^{\infty}}r_{j}\left(
t\right)  \left[  T\left(  a_{1}+x_{j}^{(1)},\ldots,a_{n}+x_{j}^{(n)}\right)
-T\left(  a_{1},...,a_{n}\right)  \right]  \right\Vert ^{2}dt\right)
^{1/2}\leq k.
\]
and using the Contraction Principle we have%
\[
(a_{1},\ldots,a_{n})\in\bigcap\limits_{m}F_{k,\left(  x_{j}^{(1)}\right)
_{j=1}^{m},\ldots,\left(  x_{j}^{(n)}\right)  _{j=1}^{m}}.
\]

On the other hand, if
\[
(a_{1},\ldots,a_{n})\in\bigcap\limits_{m}F_{k,\left(  x_{j}^{(1)}\right)
_{j=1}^{m},\ldots,\left(  x_{j}^{(n)}\right)  _{j=1}^{m}},
\]
then
\begin{align}
&  \lim_{m\rightarrow\infty}\left(  {\int_{0}^{1}}\left\Vert {\sum
\limits_{j=1}^{m}}r_{j}\left(  t\right)  \left[  T\left(  a_{1}+x_{j}%
^{(1)},\ldots,a_{n}+x_{j}^{(n)}\right)  -T\left(  a_{1},...,a_{n}\right)
\right]  \right\Vert ^{2}dt\right)  ^{1/2}\nonumber\\
&  =\left(  {\int_{0}^{1}}\left\Vert {\sum\limits_{j=1}^{\infty}}r_{j}\left(
t\right)  \left[  T\left(  a_{1}+x_{j}^{(1)},\ldots,a_{n}+x_{j}^{(n)}\right)
-T\left(  a_{1},...,a_{n}\right)  \right]  \right\Vert ^{2}dt\right)
^{1/2}\nonumber
\end{align}
and since
\[
\left(  {\int_{0}^{1}}\left\Vert {\sum\limits_{j=1}^{m}}r_{j}\left(  t\right)
\left[  T\left(  a_{1}+x_{j}^{(1)},\ldots,a_{n}+x_{j}^{(n)}\right)  -T\left(
a_{1},...,a_{n}\right)  \right]  \right\Vert ^{2}dt\right)  ^{1/2}\leq k,
\]
for all $m\in\mathbb{N}$, we conclude that
\[
(a_{1},\ldots,a_{n})\in F_{k,\left(  x_{j}^{(1)}\right)  _{j=1}^{\infty
},\ldots,\left(  x_{j}^{(n)}\right)  _{j=1}^{\infty}}.
\]
So the set $F_{k,\left(  x_{j}^{(1)}\right)  _{j=1}^{\infty},\ldots,\left(
x_{j}^{(n)}\right)  _{j=1}^{\infty}}$ is closed.

Now, consider the set
\[
F_{k}:=\bigcap F_{k,\left(  x_{j}^{(1)}\right)  _{j=1}^{\infty},\ldots,\left(
x_{j}^{(n)}\right)  _{j=1}^{\infty}},
\]
where $\left(  x_{j}^{(r)}\right)  _{j=1}^{\infty}\in B_{\ell_{p}^{u}(E_{r}%
)},r=1,\ldots,n$. From Lemma \ref{tr} it is obvious that
\[
E_{1}\times\cdots\times E_{n}=\bigcup\limits_{k\in\mathbb{N}}F_{k}.
\]
Using Baire Category Theorem we know that there is $k_{0}$ so that $F_{k_{0}}$
has an interior point and repeating the proof of \cite[Prop 9.3]{BBJP} (or
\cite[Theorem 4.1]{Scand}) we get the result.

$(b)\Rightarrow(c)$ is immediate.

$(c)\Rightarrow(d)$ is also simple.

$(d)\Rightarrow(a)$. We prove the case $n=2;$ the other cases are similar. We
shall prove that
\[
\left(  T\left(  a_{1}+x_{j}^{(1)},a_{2}+x_{j}^{(2)}\right)  -T\left(
a_{1},a_{2}\right)  \right)  _{j=1}^{\infty}\in Rad(F),
\]
for all $(a_{1},a_{2})\in E_{1}\times E_{2}$ and $(x_{j}^{(i)})_{j=1}^{\infty
}\in\ell_{p}^{u}(E_{i})$, $i=1,2$.

For $a_{1}=a_{2}=0,$ using (\ref{ptr}) we can conclude that the sequence
\[
S_{m}(.)={\sum\limits_{j=1}^{m}}r_{j}\left(  .\right)  T\left(  x_{j}%
^{(1)},x_{j}^{(2)}\right)
\]
is a Cauchy sequence on $L_{2}\left(  [0,1],F\right)  $, and hence
\begin{equation}
\left(  T\left(  x_{j}^{(1)},x_{j}^{(2)}\right)  \right)  _{j=1}^{\infty}\in
Rad(F). \label{finito1-->QuaseSomante}%
\end{equation}

For $a_{1}=0$ and $a_{2}\neq0$, we have
\[
\left(  T\left(  0+x_{j}^{(1)},a_{2}+x_{j}^{(2)}\right)  -T(0,a_{2})\right)
_{j=1}^{\infty}=\left(  T\left(  x_{j}^{(1)},a_{2}\right)  +T\left(
x_{j}^{(1)},x_{j}^{(2)}\right)  \right)  _{j=1}^{\infty}.
\]
From our hypothesis, we have
\begin{align*}
&  \left(  {\int_{0}^{1}}\left\Vert {\sum\limits_{j=1}^{m}}r_{j}\left(
t\right)  \left(  T\left(  0+x_{j}^{(1)},a_{2}+x_{j}^{(2)}\right)
-T(0,a_{2})\right)  \right\Vert ^{2}dt\right)  ^{1/2}\\
&  \leq C\left(  \left\Vert \left(  x_{j}^{(1)}\right)  _{j=1}^{m}\right\Vert
_{w,p}\right)  \left(  \left\Vert a_{2}\right\Vert +\left\Vert \left(
x_{j}^{(2)}\right)  _{j=1}^{m}\right\Vert _{w,p}\right)
\end{align*}
and hence the sequence $S_{m}(.)={\sum\limits_{j=1}^{m}}r_{j}\left(  .\right)
\left(  T\left(  x_{j}^{(1)},a_{2}\right)  +T\left(  x_{j}^{(1)},x_{j}%
^{(2)}\right)  \right)  $ is Cauchy on $L_{2}\left(  [0,1],F\right)  ;$ we
thus conclude that
\begin{equation}
\left(  T\left(  x_{j}^{(1)},a_{2}\right)  +T\left(  x_{j}^{(1)},x_{j}%
^{(2)}\right)  \right)  _{j=1}^{\infty}\in Rad(F). \label{kkhhg}%
\end{equation}
Now, from (\ref{finito1-->QuaseSomante}) and (\ref{kkhhg}) we get%
\[
\left(  T\left(  x_{j}^{(1)},a_{2}\right)  \right)  _{j=1}^{\infty}\in
Rad(F).
\]

In a similar way we conclude that%
\[
\left(  T\left(  a_{1},x_{j}^{(2)}\right)  \right)  _{j=1}^{\infty}\in Rad(F)
\]

and finally
\begin{align}
&  \left(  T\left(  a_{1}+x_{j}^{(1)},a_{2}+x_{j}^{(2)}\right)  -T\left(
a_{1},a_{2}\right)  \right)  _{j=1}^{\infty}\nonumber\\
&  =\left(  T\left(  x_{j}^{(1)},x_{j}^{(2)}\right)  \right)  _{j=1}^{\infty
}+\left(  T\left(  x_{j}^{(1)},a_{2}\right)  \right)  _{j=1}^{\infty}+\left(
T\left(  a_{1},x_{j}^{(2)}\right)  \right)  _{j=1}^{\infty}\in Rad(F).
\end{align}

\end{proof}

\begin{remark}
\label{yb}When $n=1$ we recover the concept of almost $p$-summing linear
operators and the infimum of the $C$ satisfying Theorem
\ref{teorema-Caracterizacao-L(al,p,ev)} coincides with the norm for the
operator ideal $\Pi_{al,p}$.
\end{remark}

\section{A well-behaved norm for $\mathcal{L}_{al,p}^{ev}(E_{1},...,E_{n};F)$}

Using the previous characterization we can consider a natural norm for
$\mathcal{L}_{al,p}^{ev}(E_{1},...,E_{n};F)$ and show that under this norm
$\mathcal{L}_{al,p}^{ev}$ is a Banach multi-ideal. The proofs of Proposition
\ref{norma-Ev-Al-P} and Proposition \ref{ImagemPhiFechada} follow the similar
results from \cite{Scand} and we omit:

\begin{proposition}
\label{norma-Ev-Al-P} The map
\[
\left\{
\begin{array}
[c]{c}%
\left\Vert \cdot\right\Vert _{al,p}^{ev}:\mathcal{L}_{al,p}^{ev}%
(E_{1},...,E_{n};F)\rightarrow\lbrack0,\infty\lbrack\\
\left\Vert T\right\Vert _{al,p}^{ev}=\left\Vert \phi_{T}\right\Vert
\end{array}
\right.
\]
is a norm on $\mathcal{L}_{al,p}^{ev}(E_{1},...,E_{n};F).$
\end{proposition}

\begin{proposition}
\label{ImagemPhiFechada} The linear map $\phi:\mathcal{L}_{al,p}^{ev}\left(
E_{1},...,E_{n};F\right)  \rightarrow\mathcal{L}\left(  G_{1},...,G_{n}%
;Rad\left(  F\right)  \right)  $ given by $\phi\left(  T\right)  =\phi_{T}$ is
injective and its range is closed in $\mathcal{L}\left(  G_{1},...,G_{n}%
;Rad\left(  F\right)  \right)  $.
\end{proposition}

So, we easily get the following result:

\begin{proposition}
The space $\mathcal{L}_{al,p}^{ev}(E_{1},...,E_{n};F)$ is complete under the
norm $\left\Vert \cdot\right\Vert _{al,p}^{ev}$.
\end{proposition}

A first step to proving that $\left(  \mathcal{L}_{al,p}^{ev},\left\Vert
.\right\Vert _{al,p}^{ev}\right)  $ is a Banach multi-ideal, is to evaluate
the norm of $A_{n}:\mathbb{K}^{n}\rightarrow\mathbb{K}$ given by $A_{n}\left(
x_{1},\ldots,x_{n}\right)  =x_{1}\ldots x_{n}.$ Using that, for every $\left(
a_{n}\right)  _{n=1}^{\infty}\in\ell_{2},$
\[
\int_{0}^{1}\left\vert \sum_{j=1}^{\infty}r_{j}\left(  t\right)
a_{j}\right\vert ^{2}dt=\sum_{j=1}^{\infty}\left\vert a_{j}\right\vert ^{2},
\]
we can show that $\left\Vert A_{n}\right\Vert _{al,p}^{ev}=1$.

\begin{theorem}
$\left(  \mathcal{L}_{al,p}^{ev},\left\Vert .\right\Vert _{al,p}^{ev}\right)
$ is a Banach multi-ideal.
\end{theorem}

\begin{proof}
Let $u_{j}\in\mathcal{L}\left(  G_{j},E_{j}\right)  $, $j=1,\ldots,n$,
$T\in\mathcal{L}_{al,p}^{ev}\left(  E_{1},\ldots,E_{n};F\right)  $ and
$w\in\mathcal{L}(F;H)$.

Note that%
\begin{align*}
&  {\sum\limits_{j=1}^{m}}r_{j}(t)\left[  w\circ T\circ\left(  u_{1}%
,\ldots,u_{n}\right)  \left(  a_{1}+x_{j}^{(1)},\ldots,a_{n}+x_{j}%
^{(n)}\right)  -w\circ T\circ\left(  u_{1},\ldots,u_{n}\right)  \left(
a_{1},\ldots,a_{n}\right)  \right] \\
&  ={\sum\limits_{j=1}^{m}}r_{j}(t)\left[  w\left(  T\left(  u_{1}%
(a_{1})+u_{1}\left(  x_{j}^{(1)}\right)  ,\ldots,u_{n}(a_{n})+u_{n}\left(
x_{j}^{(n)}\right)  \right)  \right)  -w\left(  T\left(  u_{1}(a_{1}%
),\ldots,u_{n}(a_{n})\right)  \right)  \right] \\
&  =w\left(  {\sum\limits_{j=1}^{m}}r_{j}(t)\left[  T\left(  u_{1}%
(a_{1})+u_{1}\left(  x_{j}^{(1)}\right)  ,\ldots,u_{n}(a_{n})+u_{n}\left(
x_{j}^{(n)}\right)  \right)  -T\left(  u_{1}(a_{1}),\ldots,u_{n}%
(a_{n})\right)  \right]  \right)
\end{align*}
So we have%
\begin{align}
&  \left\Vert \left(  w\circ T\circ\left(  u_{1},\ldots,u_{n}\right)  \left(
a_{1}+x_{j}^{(1)},\ldots,a_{n}+x_{j}^{(n)}\right)  -w\circ T\circ\left(
u_{1},\ldots,u_{n}\right)  \left(  a_{1},\ldots,a_{n}\right)  \right)
_{j=1}^{m}\right\Vert _{Rad(H)}\nonumber\\
=  &  \left(  {\int_{0}^{1}}\left\Vert w\left(  {\sum\limits_{j=1}^{m}}%
r_{j}(t)\left[
\begin{array}
[c]{c}%
T\left(  u_{1}(a_{1})+u_{1}\left(  x_{j}^{(1)}\right)  ,\ldots,u_{n}%
(a_{n})+u_{n}\left(  x_{j}^{(n)}\right)  \right) \\
-T\left(  u_{1}(a_{1}),\ldots,u_{n}(a_{n})\right)
\end{array}
\right]  \right)  \right\Vert ^{2}dt\right)  ^{\frac{1}{2}}\nonumber\\
\leq &  \left(  \left\Vert w\right\Vert ^{2}{\int_{0}^{1}}\left\Vert
{\sum\limits_{j=1}^{m}}r_{j}(t)\left[
\begin{array}
[c]{c}%
T\left(  u_{1}(a_{1})+u_{1}\left(  x_{j}^{(1)}\right)  ,\ldots,u_{n}%
(a_{n})+u_{n}\left(  x_{j}^{(n)}\right)  \right) \\
-T\left(  u_{1}(a_{1}),\ldots,u_{n}(a_{n})\right)
\end{array}
\right]  \right\Vert ^{2}dt\right)  ^{\frac{1}{2}}\nonumber\\
\leq &  \left\Vert w\right\Vert \left\Vert T\right\Vert _{al,p}^{ev}\left(
\left\Vert u_{1}(a_{1})\right\Vert +\left\Vert \left(  u_{1}\left(
x_{j}^{(1)}\right)  \right)  _{j=1}^{m}\right\Vert _{w,p}\right)
\cdots\left(  \left\Vert u_{n}(a_{n})\right\Vert +\left\Vert \left(
u_{n}\left(  x_{j}^{(n)}\right)  \right)  _{j=1}^{m}\right\Vert _{w,p}\right)
\nonumber\\
\leq &  \left\Vert w\right\Vert \left\Vert T\right\Vert _{al,p}^{ev}\left\Vert
u_{1}\right\Vert \cdots\left\Vert u_{n}\right\Vert \left(  \left\Vert
a_{1}\right\Vert +\left\Vert \left(  x_{j}^{(1)}\right)  _{j=1}^{m}\right\Vert
_{w,p}\right)  \cdots\left(  \left\Vert a_{n}\right\Vert +\left\Vert \left(
x_{j}^{(n)}\right)  _{j=1}^{m}\right\Vert _{w,p}\right) \nonumber
\end{align}

and it follows that
\[
\left\Vert w\circ T\circ\left(  u_{1},\ldots,u_{n}\right)  \right\Vert
_{al,p}^{ev}\leq\left\Vert w\right\Vert \left\Vert T\right\Vert _{al,p}%
^{ev}\left\Vert u_{1}\right\Vert \cdots\left\Vert u_{n}\right\Vert .
\]
The other properties are easily verified.
\end{proof}

\section{The space $\left(  \mathcal{P}_{al,p}^{ev}\left(  ^{n}E;F\right)
,\left\Vert \cdot\right\Vert _{al,p}^{ev}\right)  $}

\bigskip In this section we use the following notations:

\begin{itemize}
\item If $P\in\mathcal{P}\left(  ^{n}E;F\right)  ,$ the unique symmetric
$n$-linear mapping associated to $P$ is denoted by $\overset{\vee}{P}.$

\item If $P\in\mathcal{P}\left(  ^{n}E;F\right)  ,$ $a\in E$ and $1\leq k\leq
n$, then $P_{a^{k}}\in\mathcal{P}\left(  ^{n-k}E;F\right)  $ is given by%
\[
P_{a^{k}}(x)=\overset{\vee}{P}(a,...,a,x,...,x).
\]

\item If $T\in\mathcal{L}\left(  ^{n}E;F\right)  $ is symmetric$,$ $a\in E$
and $1\leq k\leq n$, then $T_{a^{k}}\in\mathcal{L}\left(  ^{n-k}E;F\right)  $
is given by%
\[
T_{a^{k}}(x_{1},...,x_{n-k})=T(a,...,a,x_{1},...,x_{n-k}).
\]
By using the Polarization Formula (see \cite[Corollary 1.6]{din}) we can
easily prove the following result:
\end{itemize}

\begin{proposition}
\bigskip$P\in\mathcal{P}_{al,p}^{ev}\left(  ^{n}E;F\right)  $ if and only if
$\overset{\vee}{P}\in\mathcal{L}_{al,p}^{ev}\left(  ^{n}E;F\right)  .$
\end{proposition}

\bigskip So, we consider the map
\[
\left\Vert \cdot\right\Vert _{al,p}^{ev}:\mathcal{P}_{al,p}^{ev}\left(
^{n}E;F\right)  \rightarrow\left[  0,+\infty\right)
\]
given by
\[
\left\Vert P\right\Vert _{al,p}^{ev}:=\left\Vert \overset{\vee}{P}\right\Vert
_{al,p}^{ev}.
\]
Since $\left(  \mathcal{L}_{al,p}^{ev},\left\Vert .\right\Vert _{al,p}%
^{ev}\right)  $ is a Banach multi-ideal, from \cite[Proposition 2.5.2]{Bra}
(see also \cite[page 46]{BBJP}) it follows that:

\begin{proposition}
\label{ghj}$\left(  \mathcal{P}_{al,p}^{ev},\left\Vert \cdot\right\Vert
_{al,p}^{ev}\right)  $ is a (Banach) polynomial ideal.
\end{proposition}

We will show that $\left(  \mathcal{P}_{al,p}^{ev},\left\Vert \cdot\right\Vert
_{al,p}^{ev}\right)  $ is a (global) holomorphy type in the sense of L. Nachbin.

An $\left(  n+1\right)  $-linear map $A\in\mathcal{L}\left(  ^{n}E,G;F\right)
$ is symmetric in the $n$\textit{ first variables if}
\[
A\left(  x_{1},\ldots,x_{n},y\right)  =A\left(  x_{\sigma(1)},\ldots
,x_{\sigma(n)},y\right)
\]
for all permutation $\sigma$ in the set $\left\{  1,\ldots,n\right\}  ,$ for
all $x_{1},\ldots,x_{n}\in E$ and $y\in G.$

Below, given $A\in\mathcal{L}\left(  E_{1},\ldots,E_{n};F\right)  $ and $a\in
E_{n}$ we define $Aa\in\mathcal{L}\left(  E_{1},\ldots,E_{n-1};F\right)  $ by
\[
Aa\left(  x_{1},\ldots,x_{n-1}\right)  =A\left(  x_{1},\ldots,x_{n-1}%
,a\right)  .
\]
The following property (Property (B)) was explored in \cite{BBJP}. The name
\textquotedblleft Property B\textquotedblright\ is dedicated to Hans-Andreas
Braunss, who essentially worked with a similar concept in his PhD thesis
\cite{Br}.

\begin{definition}
[Property (B)]Let $\mathcal{I}$ be a class of continuous multilinear maps
between Banach spaces such that for all positive integer $n$ and Banach spaces
$E_{1},\ldots,E_{n}$, $F$ the \ component
\[
\mathcal{I}\left(  E_{1},\ldots,E_{n};F\right)  :=\mathcal{L}\left(
E_{1},\ldots,E_{n};F\right)  \cap\mathcal{I}%
\]
is a linear subspace of $\mathcal{L}\left(  E_{1},\ldots,E_{n};F\right)  $
endowed with a norm represented by $\Vert\cdot\Vert_{\mathcal{I}}$. The class
$\mathcal{I}$ has the property $\left(  B\right)  $ if there is a $C\geq1$ so
that for all $n\in\mathbb{N}$, all Banach spaces $E$ and $F$ and all
$A\in\mathcal{I}\left(  ^{n}E,\mathbb{K};F\right)  $ which is symmetric in the
$n$ first variables, it occurs
\[
A1\in\mathcal{I}\left(  ^{n}E;F\right)  \text{ and }\Vert A1\Vert
_{\mathcal{I}}\leq C\Vert A\Vert_{\mathcal{I}}%
\]

\end{definition}

\begin{theorem}
\label{gyt}$\left(  \mathcal{P}_{al,p}^{ev},\left\Vert .\right\Vert
_{al,p}^{ev}\right)  $ is a global holomorphy type.
\end{theorem}

\begin{proof}
Standard computations show that $\left(  \mathcal{L}_{al,p}^{ev},\left\Vert
.\right\Vert _{al,p}^{ev}\right)  $ has property (B) and the result follows
from \cite[Theorem 3.2]{BBJP}.
\end{proof}

Now we show that, contrary to the original definition from \cite{Botelho-Nach}%
, our approach is coherent and compatible with the notion of almost
$p$-summing linear operators.

\begin{definition}
[Carando, Dimant and Muro \cite{CDM09}]\label{IdeaisCompativeis}Let
$\mathcal{U}$ be a normed ideal of linear operators. A normed ideal of
$n$-homogeneous polynomials $\mathcal{U}_{n}$ is compatible with $\mathcal{U}$
if there exist positive constants $A$ and $B$ such that for every Banach
spaces $E$ and $F$, the following conditions hold:

$\left(  i\right)  $ For each $P\in\mathcal{U}_{n}\left(  E;F\right)  $ and
$a\in E$, $P_{a^{n-1}}$ belongs to $\mathcal{U}\left(  E;F\right)  $ and
\[
\left\Vert P_{a^{n-1}}\right\Vert _{\mathcal{U}\left(  E;F\right)  }\leq
A\left\Vert P\right\Vert _{\mathcal{U}_{n}\left(  E;F\right)  }\left\Vert
a\right\Vert ^{n-1}.
\]

$\left(  ii\right)  $ For each $T\in\mathcal{U}\left(  E;F\right)  $ and
$\gamma\in E^{\ast}$, $\gamma^{n-1}T$ belongs to $\mathcal{U}_{n}\left(
E;F\right)  $ and
\[
\left\Vert \gamma^{n-1}T\right\Vert _{\mathcal{U}_{n}\left(  E;F\right)  }\leq
B\left\Vert \gamma\right\Vert ^{n-1}\left\Vert T\right\Vert _{\mathcal{U}%
\left(  E;F\right)  }%
\]

\end{definition}

\begin{definition}
[Carando, Dimant and Muro \cite{CDM09}]\label{IdeaisCoerentes}Consider the
sequence $\left\{  \mathcal{U}_{k}\right\}  _{k=1}^{N}$, where for each $k$,
$\mathcal{U}_{k}$ is a ideal of $k$-homogeneous polynomials and $N$ is
eventually infinite. The sequence $\left\{  \mathcal{U}_{k}\right\}  _{k}$ is
a coherent sequence of polynomial ideals if there exist positive constants $C$
and $D$ such that for every Banach spaces $E$ and $F$, the following
conditions hold for $k=1,\ldots,N-1$:

$\left(  i\right)  $ For each $P\in\mathcal{U}_{k+1}\left(  E;F\right)  $ and
$a\in E$, $P_{a}$ belongs to $\mathcal{U}_{k}\left(  E;F\right)  $ and
\[
\left\Vert P_{a}\right\Vert _{\mathcal{U}_{k}\left(  E;F\right)  }\leq
C\left\Vert P\right\Vert _{\mathcal{U}_{k+1}\left(  E;F\right)  }\left\Vert
a\right\Vert .
\]

$\left(  ii\right)  $ For each $P\in\mathcal{U}_{k}\left(  E;F\right)  $ and
$\gamma\in E^{\ast}$, $\gamma P$ belongs to $\mathcal{U}_{k+1}\left(
E;F\right)  $ and
\[
\left\Vert \gamma P\right\Vert _{\mathcal{U}_{k+1}}\left(  E;F\right)  \leq
D\left\Vert \gamma\right\Vert \left\Vert P\right\Vert _{\mathcal{U}_{k}\left(
E;F\right)  }.
\]

\end{definition}

\bigskip

In our next two results we will represent $\left(  \mathcal{P}_{al,p}%
^{ev}\left(  ^{n}E;F\right)  ,\left\Vert \cdot\right\Vert _{al,p}^{ev}\right)
$ by $\left(  \mathcal{P}_{al,p}^{ev(n)}\left(  ^{n}E;F\right)  ,\left\Vert
\cdot\right\Vert _{al,p}^{ev(n)}\right)  $ in order to be more precise.

\begin{theorem}
\label{PolinomioCoerente}The sequence $\left\{  \mathcal{P}_{al,p}%
^{ev(k)},\left\Vert \cdot\right\Vert _{al,p}^{ev(k)}\right\}  _{k=1}^{N}$ is coherent.
\end{theorem}

\begin{proof}
Let $a,b\in E$ and $P\in\mathcal{P}_{al,p}^{ev(k+1)}\left(  ^{k+1}E;F\right)
$. From the Polarization Formula one can prove that
\[
\left(  P_{a}\right)  ^{\vee}=\overset{\vee}{P}_{a}%
\]
and from the definition of the norms $\left\Vert .\right\Vert _{al,p}%
^{ev(k)},$ a simple calculation shows that%
\[
\left\Vert \overset{\vee}{P}_{a}\right\Vert _{al,p}^{ev(k)}\leq\left\Vert
\overset{\vee}{P}\right\Vert _{al,p}^{ev(k+1)}\left\Vert a\right\Vert .
\]
So we obtain
\[
\left\Vert P_{a}\right\Vert _{al,p}^{ev(k)}=\left\Vert \left(  P_{a}\right)
^{\vee}\right\Vert _{al,p}^{ev(k)}=\left\Vert \overset{\vee}{P}_{a}\right\Vert
_{al,p}^{ev(k)}\leq\left\Vert \overset{\vee}{P}\right\Vert _{al,p}%
^{ev(k+1)}\left\Vert a\right\Vert =\left\Vert P\right\Vert _{al,p}%
^{ev(k+1)}\left\Vert a\right\Vert
\]
and $\left(  i\right)  $ of Definition \ref{IdeaisCoerentes} is satisfied.

Now, consider $\gamma\in E^{\ast}$ and $P\in\mathcal{P}_{al,p}^{ev(k)}\left(
^{k}E;F\right)  $. By invoking the Polarization Formula we have
\[
\left(  \gamma P\right)  ^{\vee}=\gamma\overset{\vee}{P}%
\]
and from the definition of the norms $\left\Vert .\right\Vert _{al,p}^{ev(k)}$
it is not very difficult to prove that%
\[
\left\Vert \gamma\overset{\vee}{P}\right\Vert _{al,p}^{ev(k+1)}\leq\left\Vert
\gamma\right\Vert \left\Vert \overset{\vee}{P}\right\Vert _{al,p}^{ev(k)}.
\]
Hence
\[
\left\Vert \gamma P\right\Vert _{al,p}^{ev(k+1)}=\left\Vert \left(  \gamma
P\right)  ^{\vee}\right\Vert _{al,p}^{ev(k+1)}=\left\Vert \gamma\overset{\vee
}{P}\right\Vert _{al,p}^{ev(k+1)}\leq\left\Vert \gamma\right\Vert \left\Vert
\overset{\vee}{P}\right\Vert _{al,p}^{ev(k)}=\left\Vert \gamma\right\Vert
\left\Vert P\right\Vert _{al,p}^{ev(k)}%
\]
and the result follows.
\end{proof}

\bigskip

\begin{theorem}
For every positive integer $k$, the polynomial ideal $\left(  \mathcal{P}%
_{al,p}^{ev(k)},\left\Vert \cdot\right\Vert _{al,p}^{ev(k)}\right)  $ is
compatible with $\Pi_{al,p}$.
\end{theorem}

\begin{proof}
Let $u\in\Pi_{al,p}\left(  E;F\right)  $ and $\gamma\in E^{\ast}$. From the
proof of the Proposition \ref{PolinomioCoerente} it follows that
\[
\gamma u\in\mathcal{P}_{al,p}^{ev(2)}\left(  ^{2}E;F\right)
\]
and
\[
\left\Vert \gamma u\right\Vert _{al,p}^{ev(2)}\leq\left\Vert \gamma\right\Vert
\left\Vert u\right\Vert _{al,p}%
\]
Analogously,
\[
\gamma^{2}u\in\mathcal{P}_{al,p}^{ev(3)}\left(  ^{3}E;F\right)
\]
and
\[
\left\Vert \gamma^{2}u\right\Vert _{al,p}^{ev(3)}\leq\left\Vert \gamma
\right\Vert \left\Vert \gamma u\right\Vert _{al,p}^{ev(2)}\leq\left\Vert
\gamma\right\Vert ^{2}\left\Vert u\right\Vert _{al}.
\]

Proceeding by induction, we have $\left(  ii\right)  $ of Definition
\ref{IdeaisCompativeis}.

Now, let $a\in E$ and $P\in\mathcal{P}_{al,p}^{ev(n)}\left(  ^{n}E;F\right)
$. Then, from the proof of the Proposition \ref{PolinomioCoerente} we have
\[
P_{a}\in\mathcal{P}_{al,p}^{ev(n-1)}\left(  ^{n-1}E;F\right)
\]

and
\[
\left\Vert P_{a}\right\Vert _{al,p}^{ev(n-1)}\leq\left\Vert P\right\Vert
_{al,p}^{ev(n)}\left\Vert a\right\Vert .
\]

The proof follows by induction.
\end{proof}

\end{document}